\documentclass[%
 aip,
 amsmath,amssymb,
 reprint,%
]{revtex4-1}

\usepackage[british]{babel}
\usepackage{amsmath,amssymb,wasysym,amsfonts,amsthm,amsthm}
\usepackage{empheq}
\usepackage{bm}

\usepackage{graphicx,epsfig}
\usepackage{subfig}
\usepackage{overpic}

\usepackage[utf8]{inputenc}
\usepackage[T1]{fontenc}
\usepackage{mathptmx}
\usepackage{etoolbox}
\usepackage{color}

\usepackage{dcolumn}
\usepackage{hyperref}
\usepackage{url,enumerate,framed,ragged2e}

\newtheorem{thm}{Theorem}
\newtheorem*{thm*}{Theorem}
\newtheorem{lem}[thm]{Lemma}

\theoremstyle{remark}
\newtheorem{rmk}[thm]{Remark}

\newcommand{\eps}{\varepsilon}

\makeatletter
\def\@email#1#2{%
 \endgroup
 \patchcmd{\titleblock@produce}
  {\frontmatter@RRAPformat}
  {\frontmatter@RRAPformat{\produce@RRAP{*#1\href{mailto:#2}{#2}}}\frontmatter@RRAPformat}
  {}{}
}%
\makeatother

\begin{document}

\preprint{AIP/123-QED}

 \title{Controlling pulse stability in singularly perturbed reaction-diffusion systems}

 \author{F. Veerman}
 \email{f.w.j.veerman@math.leidenuniv.nl}
 \affiliation{Leiden University, Mathematisch Instituut, Niels Bohrweg 1, 2333 CA Leiden, The Netherlands}

 \author{I. Schneider}
  \email{isabelle.schneider@fu-berlin.de}
 \altaffiliation[Also at ]{Freie Universität Berlin, Institut für Mathematik, Arnimallee 7, 14195 Berlin, Germany}

\affiliation{ 
Universität Rostock, Institut für Mathematik, Ulmenstr. 69, 18057 Rostock, Germany
}%

 \date{\today}

\begin{abstract} 
\noindent 
The aim of this paper is to investigate the use of Pyragas control on the stability of stationary, localised coherent structures in a general class of two-component, singularly perturbed, reaction-diffusion systems.
We use noninvasive Pyragas-like proportional feedback control to stabilise a singular pulse solution to a two-component, singularly perturbed reaction-diffusion system. We show that in a significant region of parameter space, the control can be adjusted to stabilise an otherwise unstable pulse.  
\end{abstract}

\maketitle


\begin{quotation}
Singularly perturbed pulses in two-component reaction-diffusion equations are generally unstable in large regions of parameter space. To address this instability, we utilize Pyragas control, which was originally developed for periodic solutions of ordinary differential equations, to create a noninvasive feedback control for singularly perturbed pulses in reaction-diffusion systems. We prove the effectiveness of this control technique by analyzing the spectral stability of the controlled singular pulse in a toy model, using Evans function techniques. 

\end{quotation}

\section{Introduction}

 Reaction-diffusion systems are mathematical models based on semi-linear parabolic partial differential equations \cite{HEN81}. One of the most fascinating aspects of reaction-diffusion equations is the plethora of patterns that can emerge from their solutions. These patterns include such diverse phenomena spiral waves in a chemical oscillator \cite{WIN84}, the different animal coat patterns \cite{MUR88},  Faraday waves \cite{CHE99}, or various patterns in Rayleigh--Bénard convection \cite{BOD00}, geology\cite{SHE07}, or fluid dynamics \cite{CRA91}. 

These patterns, of which some can be classified as Turing patterns\cite{TUR52}, exhibit a variety of behaviors, including the formation of travelling waves and wave-like phenomena, as well as self-organized structures such as stripes, hexagons, and dissipative solitons \cite{HOY06}. The study of these patterns has significant implications for understanding natural phenomena and has led to new insights in various fields\cite{GOL03}.

The focus of this article is on the analysis and control of a specific spatially localized coherent structure in a two-component reaction-diffusion equation: a symmetric singular pulse \cite{VeermanDoelman.2013,DoelmanVeerman.2015,DGK01}. 
 An illustration of this pulse is presented in Figure \ref{fig:singular_pulse}. It is characterized by a noticeable scale separation, which is evident in the difference in pulse width between the two components.

It has been shown\cite{DGK01,DoelmanVeerman.2015} that such two-component pulses can only be stable when the nonlinear interaction between the two components is sufficiently strong. Even when this is the case, singular pulses are unstable for large regions in parameter space, see Ref \onlinecite[Lemmas 5.11, 5.12, and 5.14]{DoelmanVeerman.2015}.
As such, singular pulses are infrequently observed in nature or experiments. Often only the progression toward a stable steady state is observable, while unstable steady states remain largely imperceptible.

It is therefore our goal to introduce a Pyragas-like control term to make unstable pulses visible. Pyragas control \cite{PYR92,PYR06} is advantageous as it is  noninvasive on the pulses, i.e., the control term vanishes and does not change the pulse itself. However, it changes the nearby solutions and thereby the stability properties of the pulse. 
Another benefit of this control approach is its model-independence and low implementation cost, as it does not require expensive calculations.
Although originally designed for periodic solutions of ordinary differential equations, we adapt it for controlling singularly perturbed pulses in reaction-diffusion systems.

We aim to integrate control theory with the theory of pattern existence and stability in singularly perturbed reaction-diffusion systems. Specifically, we aim to control the stability of pulse solutions and develop a novel methodology to regulate the stability of diverse patterns in a broad class of singularly perturbed reaction-diffusion systems.


This paper is organized as follows: In Section \ref{modelpatterncontrol}, we discuss our model system and introduce the singular pulse as the pattern of interest; we also introduce noninvasive control terms.
Next, in Section \ref{evans}, we construct an Evans function to determine the spectral stability of the singular pulse.
Our main result is presented in Section \ref{stabilization}, where we demonstrate that noninvasive feedback stabilization can be achieved in a large region of parameter space.  We provide a brief summary of the stability proof in this section, with a complete and detailed version of the proof available in the appendix (section \ref{app:proof}).
We conclude with a short discussion in Section \ref{conclusion}.

\section{Model, patterns and control}\label{modelpatterncontrol}
\subsection{Model}
We consider the following general class of two-component, singularly perturbed, reaction-diffusion equations:
\begin{subequations}\label{eq:rdsystem_general}
  \begin{align}
     u_t &= u_{xx} - \mu u + F_1(u) + \frac{1}{\eps}F_2(u,v),\\
     v_t &= \eps^2 v_{xx} -v + G(u,v),
 \end{align}
\end{subequations}
with $x \in \mathbb{R}$, $\mu>0$ and $0<\eps \ll 1$ asymptotically small. The nonlinear functions $F_{1,2}$ and $G$ obey mild regularity assumptions \cite{DoelmanVeerman.2015}; most importantly, $F_2$ and $G$ converge superlinearly to $0$ as $V \to 0$.\\
All subsequent calculations will be carried out in the context of the following toy model:
\begin{subequations}\label{eq:rdsystem_explicit}
  \begin{align}
     u_t &= u_{xx} - u + \frac{1}{\eps}f(u)^2 T_o(u) \,\frac{v^2}{3},\\
     v_t &= \eps^2 v_{xx} - v + f(u) v^2,
 \end{align}
\end{subequations}
with $f$ and $T_0$ smooth functions of $u$.\\
System \eqref{eq:rdsystem_explicit} exhibits most defining qualities of the general model class \eqref{eq:rdsystem_general}, while allowing for explicit computations. We emphasise, however, that the methods and techniques used in this paper apply to all systems of the general class \eqref{eq:rdsystem_general}.


\subsection{Patterns}
The spatially localised coherent structure to be studied in this paper is a symmetric singular pulse, bi-asymptotic to the trivial background state; see Figure \ref{fig:singular_pulse}. The singularly perturbed nature of \eqref{eq:rdsystem_general}, through the asymptotically small diffusion term $\eps^2 v_{xx}$, induces a spatial scale separation in stationary pattern solutions to \eqref{eq:rdsystem_general}. For singular pulses, this scale separation is visible in the difference in pulse width between the $u$- and $v$-component, see again Figure \ref{fig:singular_pulse}; to this end, we introduce the short-scale spatial variable
\begin{equation}
\xi = \frac{x}{\eps}.
\end{equation}
The presence of the asymptotically small parameter $\eps$ allows for the application of Geometric Singular Perturbation Theory (GSPT) to rigorously establish the existence --by a constructive proof-- of stationary, symmetric singular pulse solutions \cite{DoelmanVeerman.2015}. Hence, we state the existence of a stationary, symmetric pulse solution $\left(u_p(x),v_p(\xi)\right)$ to \eqref{eq:rdsystem_general}, with $\left(u_p(x),v_p(\xi)\right) \sim \left(e^{-|x|}, e^{-|\xi|}\right)$ as $x \to \pm \infty$, provided an algebraic condition in terms of $F_{1,2}$ and $G$ is satisfied\cite{DoelmanVeerman.2015}. For the toy model \eqref{eq:rdsystem_explicit}, this algebraic condition is
\begin{equation}\label{eq:pulse_existence_condition}
    \mu u^2 = T_o(u)^2.
\end{equation}
In addition,
\begin{equation}\label{eq:pulse_toymodel}
    \left(u_p(x),v_p(x/\eps)\right) = \left(u_* e^{-|x|}, \frac{3}{2 f(u_*)} \text{sech}^2\frac{x}{2\eps}\right) + \mathcal{O}(\eps),
\end{equation}
where $u_*>0$ is a nondegenerate solution to \eqref{eq:pulse_existence_condition}, for which we assume that $f(u_*) \neq 0$. For more details on the pulse construction and proof of existence, see Ref \onlinecite{DoelmanVeerman.2015}.

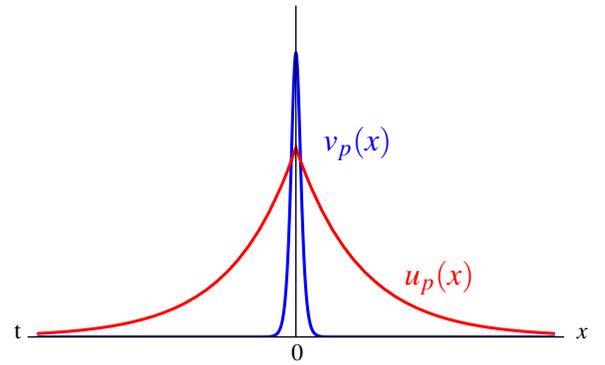
\begin{figure}{t}
    \begin{overpic}[width=0.4\textwidth]{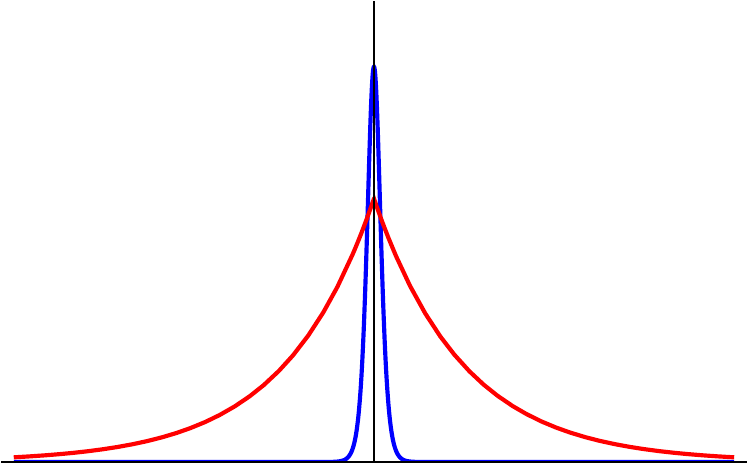}
	\put(70,10){\large{\textcolor{red}{$u_p(x)$}}}
    \put(55,35){\large{\textcolor{blue}{$v_p(x)$}}}
    \put(49,-4){$0$}
    \put(102,0){$x$}
	  \end{overpic}
    \caption{A typical profile of a singular symmetric pulse, bi-asymptotic to the trivial background state. A spatial scale separation between the large scale component $u$ (red) and the small scale component $v$ (blue) is clearly visible.}
    \label{fig:singular_pulse}
\end{figure}


\subsection{Control}
The aim of this paper is to investigate the use of Pyragas control \cite{PYR92,PYR06,SOC94} on the stability of stationary, localised coherent structures in \eqref{eq:rdsystem_general}. To that end, we introduce feedback control terms $K$ and $L$ to \eqref{eq:rdsystem_general}, yielding
\begin{subequations}\label{eq:rdsystem_general_control}
  \begin{align}
     u_t &= u_{xx} - \mu u + F_1(u) + \frac{1}{\eps}F_2(u,v) + K[u,v],\\
     v_t &= \eps^2 v_{xx} -v + G(u,v) + L[u,v],
 \end{align}
\end{subequations}
 where $K[u_p(x),v_p(\xi)] = 0 = L[u_p(x),v_p(\xi)]$. As the feedback control terms vanish on the target pattern $\left(u_p,v_p\right)$, this type of control is called \emph{non-invasive}. For control of patterns in PDEs, a wide range of possible control terms can be applied, including spatio-temporal delay, proportional feedback, or combinations of these \cite{LU96, MON04,Schneider.2017,SCH22}
 In this paper, we investigate one specific form of non-invasive control, applied to the toy model \eqref{eq:rdsystem_explicit}. We leave the $u$-equation intact, and add proportional feedback control to the $v$-equation, yielding
\begin{subequations}\label{eq:rdsystem_explicit_control}
  \begin{align}
     u_t &= u_{xx} - u + \frac{1}{\eps}f(u)^2 T_o(u) \,\frac{v^2}{3},\\
     v_t &= \eps^2 v_{xx} - v + f(u) v^2 + \ell\left(v-v_p(x)\right),
 \end{align}
\end{subequations}
 with $\ell:\mathbb{R} \to \mathbb{R}$ continuously differentiable at zero, and $\ell(0)=0$. Note that the latter condition ensures that $\left(u_p,v_p\right)$ \eqref{eq:pulse_toymodel} is a solution to \eqref{eq:rdsystem_explicit_control}. We have opted for control solely on the variable $v$ since it exhibits a high degree of spatial localization, with the $v$-component of the pulse essentially being zero beyond the spike region. This offers advantages from an application standpoint, as implementing localized controls is often simpler compared to extended controls covering a larger area.
 
 Note that, consequently, on the linear level, the control is not a multiple of the identity matrix. Therefore, the influence of the control term $\ell$ is not as straightforward as shifting all eigenvalues to the left.
 Our goal is to derive conditions on the control function $\ell$ such that the singular pulse $\left(u_p,v_p\right)$ \eqref{eq:pulse_toymodel} is a stable solution to \eqref{eq:rdsystem_explicit_control}.
 



\section{Pulse stability and the Evans function}\label{evans}
In order to be observable, a stationary pulse solution needs to be \emph{stable} as a solution to the PDE system \eqref{eq:rdsystem_general}. The (spectral) stability of a singular pulse solution to \eqref{eq:rdsystem_general} can be determined through constructing an Evans function $\mathcal{E}(\lambda)$, an analytic function whose roots precisely coincide with the (discrete) spectrum of the linear operator obtained by linearising \eqref{eq:rdsystem_general} at the singular pulse $(u_p,v_p)$ \cite{Sandstede.2002,KP.2013}. That is, we consider the eigenvalue problem
\begin{equation}\label{eq:evproblem_nocontrol}
    \mathcal{L} \begin{pmatrix} u\\ v \end{pmatrix} = \lambda \begin{pmatrix} u\\ v \end{pmatrix},
\end{equation}
with
\begin{equation}\label{eq:operator_L}
  \mathcal{L} = \begin{pmatrix} \partial_x^2 - 1 + F_1' + \frac{1}{\eps} \frac{\partial F_2}{\partial u} & \frac{1}{\eps} \frac{\partial F_2}{\partial v} \\ \frac{\partial G}{\partial u} & \eps^2 \partial_x^2 - 1 + \frac{\partial G}{\partial v}
  \end{pmatrix},
\end{equation}
where all (partial) derivatives of $F_{1,2}$ and $G$ are evaluated at $(u,v) = (u_p(x),v_p(\xi))$.\\
The essential spectrum of $\mathcal{L}$ \eqref{eq:operator_L} is real, negative, and bounded away from the imaginary axis\cite{DoelmanVeerman.2015}; hence, the pulse stability is determined by its discrete spectrum, i.e. the roots of the associated Evans function $\mathcal{E}(\lambda)$. In Ref \onlinecite{DoelmanVeerman.2015}, it is shown that the singularly perturbed structure of the pulse can be used to obtain an explicit characterisation of the roots of $\mathcal{E}(\lambda)$ to leading order in $\eps$; moreover, these roots perturb regularly in $\eps$.\\ For sake of brevity, we omit further details, and only state the main outcome of the theory developed in Ref \onlinecite{DoelmanVeerman.2015} when applied to the toy model \eqref{eq:rdsystem_explicit}: the spectrum of the pulse \eqref{eq:pulse_toymodel} is to leading order in $\eps$ determined by the roots of the function
\begin{eqnarray}\label{eq:ts_nocontrol}
    t_s(\lambda) := :=&& \frac{T_o'(u_*)}{T_o(u_*)} - \frac{1}{u_*}\sqrt{1+\lambda} \\
        &&+ \frac{f'(u_*)}{f(u_*)}\left(2  + \frac{1}{3} \int_{-\infty}^\infty \!\!\!\!\!\hat{v}_p(\xi) \hat{v}_\text{in}(\xi;\lambda)\,\text{d}\xi\right),\nonumber
\end{eqnarray}        
where $\hat{v}_p(\xi) = f(u_*) v_p(\xi)$ and $\hat{v}_\text{in}$ is the unique bounded solution to 
\begin{equation}\label{eq:vinhat_eq}
 \left[\partial_\xi^2 - 1-\lambda + 2 \hat{v}_p(\xi)\right]\hat{v}_\text{in} = - \hat{v}_p(\xi)^2,
\end{equation}
which does not depend on $u_*$.\\

A typical configuration of the spectrum of $\mathcal{L}$ \eqref{eq:operator_L} is shown in Figure \ref{fig:spectrum_sketch}. The pulse can lose stability when a pair of eigenvalues crosses the imaginary axis (a Hopf bifurcation) or when a real eigenvalue passes through the origin. Note that the pulse spectrum consists of both discrete (point) spectrum and essential (continuous) spectrum, as the spatial domain is unbounded.

\begin{figure}{t}
    \begin{overpic}[width=0.4\textwidth]{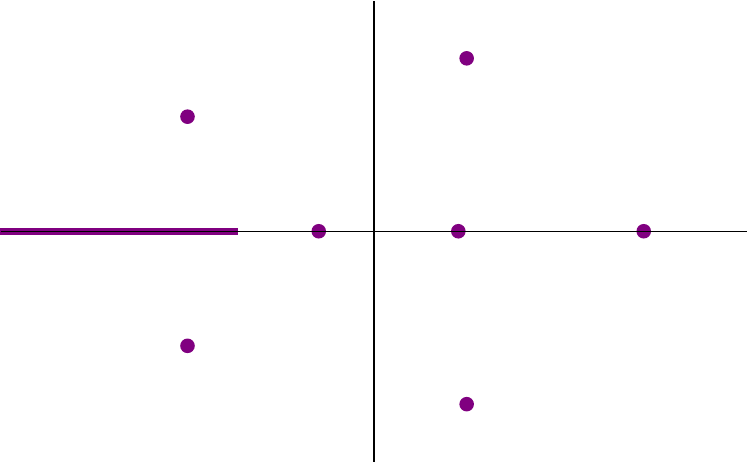}
	\put(95,55){\large{$\mathbb{C}$}}
    \put(102,30){Re $\lambda$}
    \put(45,63){Im $\lambda$}
    \end{overpic}
    \caption{A typical configuration of the spectrum of the pulse (purple) in the complex plane. Spectrum to the left of the imaginary axis is stable. The essential spectrum is seen to be real, negative, and bounded away from the imaginary axis.}
    \label{fig:spectrum_sketch}
\end{figure}



\section{Pulse stabilisation through proportional feedback control}\label{stabilization}
The main research question that we address in this paper is: Given a singular pulse solution $(u_p,v_p)$ to \eqref{eq:rdsystem_general}, can we find control terms $K$, $L$ such that this singular pulse is a stable solution to \eqref{eq:rdsystem_general_control}?\\
As the introduction of control terms has a (potentially) significant influence on the stability analysis of the singular pulse, we try to answer the research question formulated above in the context of the toy problem \eqref{eq:rdsystem_explicit}, with proportional feedback control in the $v$-equation \eqref{eq:rdsystem_explicit_control}. We first present the main outcome of our analysis in Theorem \ref{thm:caseIa}, and describe the main ideas of the proof. The full proof of Theorem \ref{thm:caseIa} can be found in Appendix \ref{app:proof}.

\begin{thm}\label{thm:caseIa}
 Let $0<\eps\ll 1$ be sufficiently small, and assume that $u_*$ is a nondegenerate solution to \eqref{eq:pulse_existence_condition}. Consider the symmetric singular pulse solution $(u_p,v_p)$ to \eqref{eq:rdsystem_explicit}, which is to leading order in $\eps$ given by \eqref{eq:pulse_toymodel}, and introduce $\rho := 2 \frac{f'(u_*)}{f(u_*)} + \frac{T_o'(u_*)}{T_o(u_*)}$.
 \begin{enumerate}
  \item If $f'(u_*) = 0$, then the singular pulse $(u_p,v_p)$ is always unstable for any choice of proportional control function $\ell(v-v_p)$ as implemented in \eqref{eq:rdsystem_explicit_control}.
  \item If $f'(u_*)<0$, then it is possible to choose a proportional control function $\ell(v-v_p)$, as implemented in \eqref{eq:rdsystem_explicit_control}, such that the singular pulse $(u_p,v_p)$ is stable.
  \item Let $f'(u_*)>0$. 
  \begin{enumerate}
  \item If $\rho > \frac{1}{u_*}$, then the singular pulse $(u_p,v_p)$ is always unstable for any choice of proportional control function $\ell(v-v_p)$ as implemented in \eqref{eq:rdsystem_explicit_control}. 
  \item If $\rho < \frac{1}{u_*}$, then it is possible to choose a proportional control function $\ell(v-v_p)$, as implemented in \eqref{eq:rdsystem_explicit_control}, such that the pulse solution $(u_p,v_p)$ is stable.
  \end{enumerate}
 \end{enumerate}
\end{thm}

A visual representation of the statement of Theorem \ref{thm:caseIa} is given in Figure \ref{fig:caseIa}. A direct application of Theorem \ref{thm:caseIa} for specific parameter values is shown in Figures \ref{fig:spectrumcontrol} and \ref{fig:spectrumcontrol2}.\\

It is worthwhile to note that control on one variable only suffices to control both components of the pulse.

Specific conditions that the control function $\ell$ needs to satisfy to stabilise the pulse can be found in the proof of Lemma \ref{lem:alphabeta_solutions}, Appendix \ref{app:proof}. In particular, the essential spectrum is stable if and only if $\ell'(0) < 1$ \eqref{eq:essspec_fulloperator}.

Moreover, note that the controllability of the pulse strongly depends on the quantity $\rho = 2(f'/f) + (T_0'/T_0)$, which corresponds to the logarithmic derivative of $f^2(u)T_0(u)$. This quantity represents the $u$-dependent nonlinearity of \eqref{eq:rdsystem_explicit}(a) at $u=u_*$.

\begin{rmk}
While the purpose of this paper is to show the stabilisation of pulse that is unstable in the absence of control, it is worthwhile to note that our control scheme can also destabilise an otherwise stable pulse. For example, choosing $\ell'(0) < 1$ destabilises the pulse through a sideband instability, as the essential spectrum is pushed through the imaginary axis \eqref{eq:essspec_fulloperator}. We do not explore such destabilisation scenarios in the current paper; the desirability of pulse destabilisation through noninvasive control depends on the model context and application.   
\end{rmk}

The proof of Theorem \ref{thm:caseIa}  starts with the observation that the linear stability of $(u_p,v_p)$ as a solution to \eqref{eq:rdsystem_explicit_control} can be written as
\begin{equation}\label{eq:evproblem_control}
    \left[\mathcal{L}  - \begin{pmatrix} 0 & 0 \\ 0 & \ell'(0)\end{pmatrix}\right]\begin{pmatrix} u\\ v \end{pmatrix} = \lambda \begin{pmatrix} u\\ v \end{pmatrix} ,
\end{equation}
with $\mathcal{L}$ given in \eqref{eq:operator_L}. This means that the procedure to construct an Evans function, as presented in Ref \onlinecite{DoelmanVeerman.2015}, can be applied to \eqref{eq:evproblem_control}. The singularly perturbed structure of the pulse again leads to the result that the solutions to the eigenvalue problem \eqref{eq:evproblem_control} are, to leading order in $\eps$, determined by the roots of the function
\begin{eqnarray}\label{eq:ts_control}
        \hat{t}_{s}(\lambda) :=&& \frac{T_o'(u_*)}{T_o(u_*)} - \frac{1}{u_*}\sqrt{1+\lambda} \\
        &&+ \frac{f'(u_*)}{f(u_*)}\left(2  + \frac{1}{3} \int_{-\infty}^\infty \!\!\!\!\!\hat{v}_p(\xi) \hat{v}_\text{in}(\xi;\lambda - \ell'(0))\,\text{d}\xi\right)\nonumber
\end{eqnarray}
The remainder of the proof is a detailed analysis of the meromorphic complex function $\hat{t}_s$. The main difficulty lies in understanding the $\lambda$-dependence of the integral term in \eqref{eq:ts_control}, which is closely related the $\lambda$-dependence of $\hat{v}_\text{in}(\lambda-\ell'(0))$, where $\hat{v}_\text{in}$ is the unique bounded solution to \eqref{eq:vinhat_eq}.\\
Next, we introduce $\hat{\lambda} = \lambda - \ell'(0)$, and use Weyl--Titchmarsh--Kodaira spectral theory\cite{Titchmarsh.1962,HigsonTan.2020} to express the integral
\begin{displaymath}
 \int_{-\infty}^\infty \!\!\!\!\!\hat{v}_p(\xi) \hat{v}_\text{in}(\xi;\hat{\lambda})\,\text{d}\xi = \left\langle \hat{v}_\text{in},\hat{v}_p\right\rangle_2
\end{displaymath}
in terms of projections onto eigenfunctions of the operator $\mathcal{L}_f := \partial_\xi^2 - 1-\hat{\lambda} + 2 \hat{v}_p(\xi)$, cf. \eqref{eq:vinhat_eq}. This allows us to obtain estimates on $\left\langle \hat{v}_\text{in},\hat{v}_p\right\rangle_2$, from which the statements of the theorem follow.

\begin{figure}[h!]
 \begin{overpic}[width=0.4\textwidth]{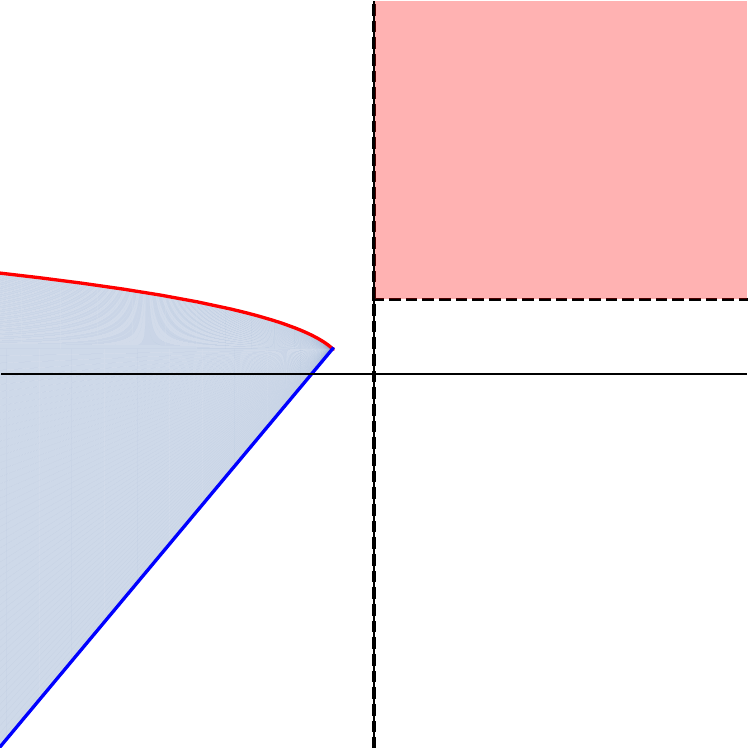}
	\put(5,40){stable}
	\put(70,80){unstable}
	\put(15,75){controllable}
	\put(65,25){controllable}
	\put(102,49){$f'(u_*)$}
	\put(48,102){$\rho$}
	\put(44,58){$\dfrac{1}{u_*}$}
	\put(46,46){$0$}
 \end{overpic}
\caption{The stability of the pulse $(u_p,v_p)$ as determined in Theorem \ref{thm:caseIa}, with $\rho = 2\frac{f'(u_*)}{f(u_*)} + \frac{T_o'(u_*)}{T_o(u_*)}$. In the blue region (numerically determined), the pulse is stable without control; stability is lost either through a Hopf bifurcation (red curve) or by a real eigenvalue passing through zero (blue curve). In the red region and on the dashed lines, the pulse is unstable and cannot be controlled, see Theorem \ref{thm:caseIa} 1) and 3a). In the remainder of the parameter space, the pulse is unstable but can be stabilised by proportional control in the $v$-component, as implemented in \eqref{eq:rdsystem_explicit_control}}\label{fig:caseIa}
\end{figure}

\begin{figure}[h!]
 \begin{overpic}[width=0.4\textwidth]{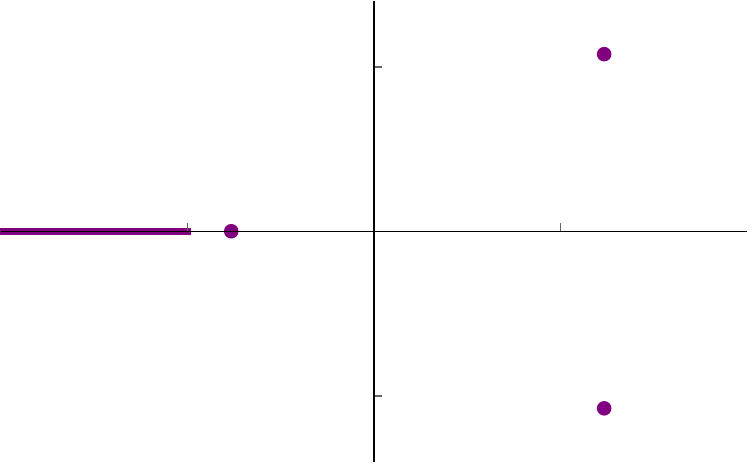}
	\put(0,60){\textbf{(A)}}
    \put(5,50){without control}
    \put(21,26){$-1$}
    \put(74,26){$1$}
    \put(46,52){$5$}
    \put(43,8){$-5$}
 \end{overpic}
 
 \vspace*{2\baselineskip}
 
 \begin{overpic}[width=0.4\textwidth]{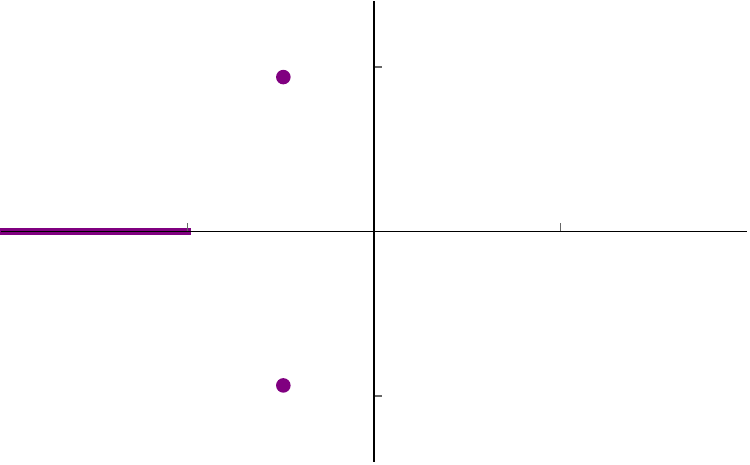}
	\put(0,60){\textbf{(B)}}
 \put(5,50){with control}
     \put(21,26){$-1$}
    \put(74,26){$1$}
    \put(46,52){$5$}
    \put(43,8){$-5$}
 \end{overpic}

  \vspace*{2\baselineskip}
 
 \begin{overpic}[width=0.4\textwidth]{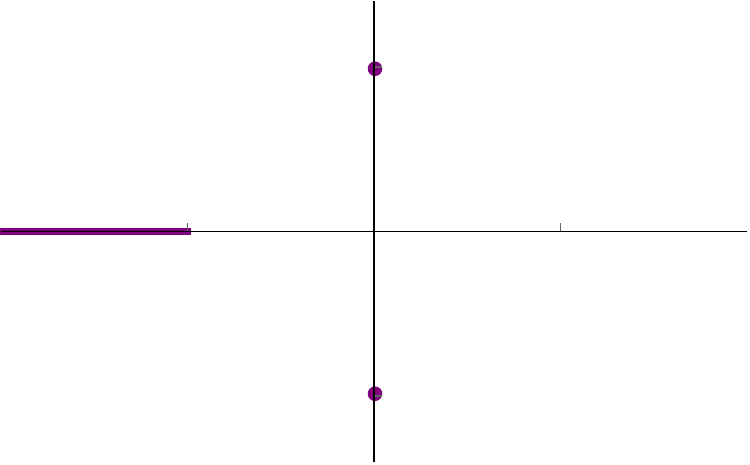}
	\put(0,60){\textbf{(C)}}
 \put(5,50){Hopf bifurcation}
     \put(21,26){$-1$}
    \put(74,26){$1$}
    \put(46,52){$5$}
    \put(43,8){$-5$}
 \end{overpic}
\caption{ \textbf{(A)} The spectrum for $(u_p,v_p)$\eqref{eq:pulse_toymodel} for the toy model \eqref{eq:rdsystem_explicit}, with parameter choices $\rho := 2 \frac{f'(u_*)}{f(u_*)} + \frac{T_o'(u_*)}{T_o(u_*)} = \frac{2}{u_*}$ and $f'(u_*) = -3\frac{f(u_*)}{u_*}$. For these parameters, the pulse is unstable. \textbf{(B)} Application of proportional feedback control as in \eqref{eq:rdsystem_explicit_control}, with $\ell'(0) = -3$ is (more than) sufficient to stabilise the pulse; this is a direct consequence of Theorem \ref{thm:caseIa}, statement 2. \textbf{(C)} Proportional feedback control as in in \eqref{eq:rdsystem_explicit_control}, with $\ell'(0) = -2.1882 + \mathcal{O}(\eps)$ stabilises the pulse through a Hopf bifurcation. The spectral configurations shown in this Figure are numerically obtained roots of the Evans function \eqref{eq:ts_control}.} \label{fig:spectrumcontrol}
\end{figure}

\begin{figure}[h!]
 
  \begin{overpic}[width=0.4\textwidth]{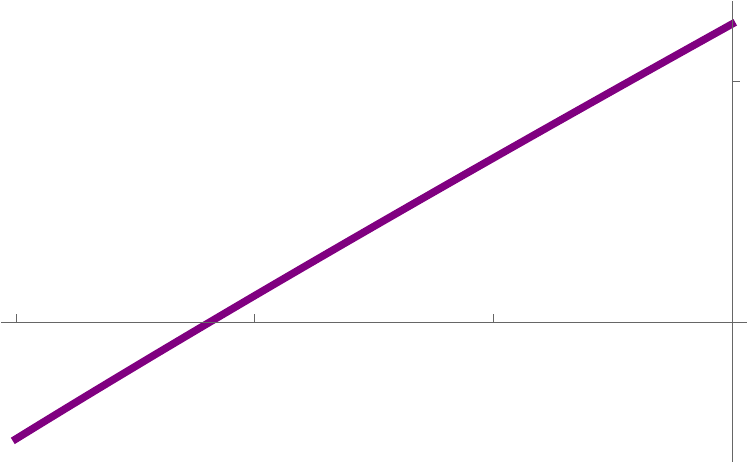}
	  \put(-2,15){$-3$}
    \put(30,15){$-2$}
    \put(62,15){$-1$}
    \put(94,49){$5$}
    \put(101,18){$\ell'(0)$}
    \put(94,64){Re $\lambda$}
 \end{overpic}

\caption{The real part of the complex eigenvalue pair shown in Figure \ref{fig:spectrumcontrol} as a function of the control strength $\ell'(0)$, with proportional feedback control as in \eqref{eq:rdsystem_explicit_control}. A sufficiently strong control will stabilise the pulse through a Hopf bifurcation.} \label{fig:spectrumcontrol2}
\end{figure}
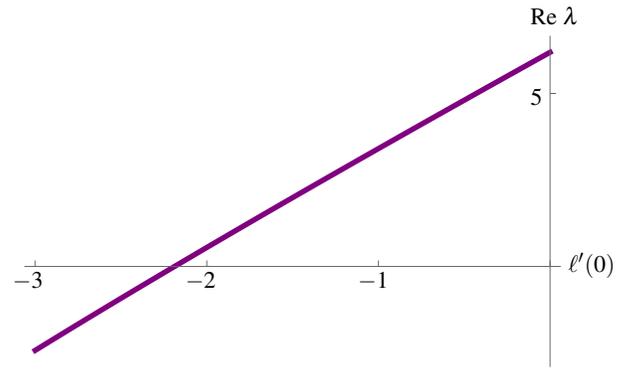

\section{Conclusion and outlook}\label{conclusion}

In conclusion, we have demonstrated the possibility of stabilizing singularly perturbed pulses in two-component reaction-diffusion equations across significant areas of the parameter plane. Specifically, for the toy model \eqref{eq:rdsystem_explicit}, we have identified three regions in the $(\rho,f'(u_*))$-plane: a stable region where control is unnecessary, an unstable region where control is always insufficient, and -- our main interest here -- a large region where pulse stability can be controlled through our proposed control term.

Several potential avenues for future investigation emerge. First, it may be worthwhile to explore the use of a ``true'' Pyragas control, in which the control is not proportional but includes a delay term, and as such does not rely on explicit pre-existing knowledge of the pulse structure. This poses additional problems in the rigorous stability analysis as difficulties may arise concerning nonlinear stability. 
For the control scheme investigated in this paper, it was not necessary to differentiate between spectral and nonlinear stability. In the non-controlled eigenvalue problem \eqref{eq:evproblem_nocontrol}, spectral stability implies nonlinear stability; this follows from a classical result by Henry\cite{HEN81} because the operator $\mathcal{L}$ \eqref{eq:operator_L} is sectorial. The linear operator associated to the 'controlled' eigenvalue problem \eqref{eq:evproblem_control} has the same (sectorial) property. Hence, one can directly infer nonlinear stability from spectral stability. For other control strategies, such as delayed feedback, the sectorial property is generally not retained. In such case, one needs to be more careful about deducing nonlinear stability; see e.g. Ref \onlinecite{Schnaubelt.2004} for generalisations of Henry's result.

Second, investigating the impact of control on the large scale component $u$ in comparison to the current focus on the small scale component $v$ would provide valuable insights for this two-component system. Additionally, exploring the potential of non-diagonal controls would also be of interest.

Third, our current analysis has focused on relatively simple spatially-localized coherent structures. It would be interesting to investigate the potential extension of our control scheme to more complex patterns, such as multi-circuit configurations\cite{DGK01} and/or periodic pulse patterns\cite{deRijk_etal.2016}, and to determine the conditions under which stabilization can be achieved for these more intricate situations. 







\begin{acknowledgments}F.V. has been supported by a Humboldt Fellowship.
I.S. has been supported by the Deutsche Forschungsgemeinschaft, SFB 910, Project A4 “Spatio-Temporal Patterns: Control, Delays, and Design”. 
We would like to thank all the members of SFB 910, in particular Sabine Klapp, Eckehard Schöll and Bernold Fiedler, for their valuable contributions and continuous support. In addition, we would like to express our gratitude to Nigel Higson for insightful discussions on Weyl's Theorem.
\end{acknowledgments}

\section*{Data Availability Statement}
Data sharing is not applicable to this article as no new data were created or analyzed in this study.

\appendix

\section{Proof of Theorem \ref{thm:caseIa}}\label{app:proof}


\subsubsection{Pulse stability and the Evans function}\label{sss:Ia_Evans} 
For the stability analysis of the pulse $(u_p,v_p)$, we follow the approach of Ref \onlinecite{DoelmanVeerman.2015}. The eigenvalue problem determining the linear stability of $(u_p, v_p)$ is given by
\begin{subequations}\label{eq:rdstability_Ia_control}
  \begin{align}
     \lambda u =& u_{xx} - u  + \frac{1}{3\eps}v_p^2\left[f(u_p)^2   T'_o(u_p)+2 f(u_p)f'(u_p)T_o(u_p)\right]u\nonumber\\
     &+ \frac{2}{3\eps}f(u_p)^2 v_p  T_o(u_p) v ,\\
     \lambda v=& \eps^2 v_{xx} - v +2 f(u_p) v_p v+ f'(u_p) v_p^2 u+\ell'(0)v.
 \end{align}
\end{subequations}
In the short scale spatial coordinate $\xi=x/\eps$, system \eqref{eq:rdstability_Ia_control} can be reformulated as
\begin{subequations}\label{eq:rdstability_Ia_control_xi}
  \begin{align}
     \eps^2(\lambda + 1)u =& u_{\xi \xi}  + \frac{\eps}{3}v_p^2\left[f(u_p)^2   T'_o(u_p)+2 f(u_p)f'(u_p)T_o(u_p)\right]u\nonumber\\&+ \frac{2 \eps}{3}f(u_p)^2 v_p  T_o(u_p) v,\\
     \lambda v=&  v_{\xi \xi} - v +2 f(u_p) v_p v+ f'(u_p) v_p^2 u+\ell'(0)v.
 \end{align}
\end{subequations}
We can write \eqref{eq:rdstability_Ia_control_xi} as a first order system for $\phi = (u,u_\xi,v,v_\xi)$, 
\begin{equation}\label{eq:rdstability_Ia_A-system}
 \phi_\xi = \mathcal{A}(\xi;\lambda,\eps)\, \phi.
\end{equation}
This formulation allows us to use the theory presented in Ref \onlinecite{Sandstede.2002}. First, we observe that the matrix $\mathcal{A}$ is asymptotically constant since $(u_p(\xi),v_p(\xi)) \to (0,0)$ as $\xi \to \pm \infty$. The associated asymptotic matrix $\mathcal{A}_\infty(\lambda,\eps) := \lim_{\xi \to \pm \infty} \mathcal{A}(\xi;\lambda,\eps)$ has eigenvalues $\left\{\pm \eps\sqrt{1+\lambda},\pm \sqrt{1+\lambda-\ell'(0)}\right\}$. Hence, $\mathcal{A}_\infty$ is hyperbolic if and only if $\lambda > -1 + \text{min}(\ell'(0),0)$. Since $\partial_\xi - \mathcal{A}$ is a relatively compact perturbation of the operator $\partial_\xi - \mathcal{A}_\infty$, the essential spectrum of $\partial_\xi - \mathcal{A}$ is given by
\begin{equation}\label{eq:essspec_fulloperator}
 \sigma_\text{ess} = \left\{\lambda \in \mathbb{R}\;:\;\lambda \leq -1+ \text{max}(\ell'(0),0)\right\}.
\end{equation}
We observe that the essential spectrum is stable if and only if $\ell'(0) < 1$. Hence, from this point onward, we assume
\begin{equation}\label{eq:rdstability_Ia_ass_l4}
 \ell'(0) < 1,
\end{equation}
as a necessary condition for the spectrum of $\partial_\xi - \mathcal{A}$ to be stable.
The remainder of the spectrum is discrete; note that elements of this point spectrum can (and generically will) have nonzero imaginary part.\\

The Evans function $\mathcal{E}(\lambda;\eps)$ can now be defined for all $\lambda \notin \sigma_\text{ess}$, as follows. Let $N(\lambda,\eps) \subset \mathbb{R}^4$ be the set of initial conditions such that the associated solutions to \eqref{eq:rdstability_Ia_A-system} decay exponentially as $\xi \to \infty$; analogously, let $R(\lambda,\eps) \subset \mathbb{R}^4$ be the set of initial conditions such that the associated solutions to \eqref{eq:rdstability_Ia_A-system} decay exponentially as $\xi \to -\infty$. Because $\mathcal{A}_\infty$ is hyperbolic for $\lambda \notin \sigma_\text{ess}$, system \eqref{eq:rdstability_Ia_A-system} has an exponential dichotomy, and the subsets $N$ and $R$ are subspaces of $\mathbb{R}^4$ of dimension 2. Choosing an ordered basis $\left\{\phi_{1,N},\phi_{2,N}\right\}$ of $N$ and an ordered basis $\left\{\phi_{1,R},\phi_{2,R}\right\}$ of $R$, we define the Evans function as
\begin{equation}
\mathcal{E}(\lambda;\eps) = \text{det} \left(\phi_{1,N},\phi_{2,N},\phi_{1,R},\phi_{2,R}\right). 
\end{equation}
The Evans function is analytic on its domain. Most importantly, $\mathcal{E}(\lambda)$ is zero if and only if $\lambda$ is an eigenvalue of \eqref{eq:rdstability_Ia_A-system}; moreover, the algebraic multiplicity of this eigenvalue is equal to the order of $\lambda$ as a zero of $\mathcal{E}$ (Ref \onlinecite[Theorem 4.1]{Sandstede.2002}).

\subsubsection{An explicit expression for the Evans function}
We can use the theory developed in Ref \onlinecite{DoelmanVeerman.2015} to determine an explicit expression for the Evans function $\mathcal{E}$. In particular, we observe that the approach and resulting expressions from Ref \onlinecite[section 4]{DoelmanVeerman.2015} can be applied directly for the shifted eigenvalue $\hat{\lambda} = \lambda - \ell'(0)$. Note that, although the eigenvalue system \eqref{eq:rdstability_Ia_control_xi} is \emph{superficially} analogous to the equivalent system (Ref \onlinecite[(3.2)]{DoelmanVeerman.2015}) with parameter $\mu = 1+\ell'(0)$, the underlying pulse solution itself depends on $\mu$ as well through the existence condition (Ref \onlinecite[(2.16)]{DoelmanVeerman.2015}). Therefore, we cannot a priori apply the instability results from Ref \onlinecite[section 5]{DoelmanVeerman.2015}. Moreover, as the stability of the controlled pulse is determined by the real part of $\lambda$, the stability condition for the shifted eigenvalue
\begin{equation}\label{eq:def_lambdahat}
\hat{\lambda}:= \lambda - \ell'(0)
\end{equation}
is given by 
\begin{equation}\label{eq:stabcond_lambdahat}
\text{Re }\hat{\lambda} < -\ell'(0).
\end{equation}
In terms of $\hat{\lambda}$, the essential spectrum is given by
\begin{equation}\label{eq:rdstability_Ia_control_essspec_hat}
 \hat{\sigma}_\text{ess} = \left\{\hat{\lambda} \in \mathbb{R}\;:\;\hat{\lambda} \leq -1+ \text{max}(0,-\ell'(0))\right\}.
\end{equation}
From the analysis in Ref \onlinecite{DoelmanVeerman.2015}, it follows that for all $\hat{\lambda} \notin \hat{\sigma}_\text{ess}$, the Evans function $\mathcal{E}$ can be written as
\begin{equation}\label{eq:Evansfunction_decomp}
 \mathcal{E}(\hat{\lambda},\eps) = 4 \eps\, t_f(\hat{\lambda},\eps)\, t_s(\hat{\lambda},\eps)\sqrt{1+\hat{\lambda}}\sqrt{1+\ell'(0) + \hat{\lambda}},
\end{equation}
in terms of the so-called transmission functions $t_f(\hat{\lambda},\eps)$ and $t_s(\hat{\lambda},\eps)$. These transmission functions are not necessarily analytic, but merely meromorphic on $\mathbb{C} \setminus \hat{\sigma}_\text{ess}$. The structure of these transmission functions can be determined in more detail, as explained in Ref \onlinecite{DoelmanVeerman.2015}. For future reference, we summarise the most important aspects of the analysis in Ref \onlinecite{DoelmanVeerman.2015} below.\\

Consider the Sturm-Liouville operator
\begin{equation}\label{eq:Lf}
 \mathcal{L}_f = \partial_\xi^2 - 1 + 2 f(u_*) v_p = \partial_\xi^2 - 1 + 3\, \text{sech}^2 \left(\frac{\xi}{2}\right),
\end{equation}
cf. \eqref{eq:pulse_toymodel}. Its spectrum is given by
\begin{equation}\label{eq:Lf_spectrum}
 \hat{\sigma}_f = \left\{\hat{\lambda}_0,\hat{\lambda}_1,\hat{\lambda}_2\right\} \cup \left(-\infty,-1\right),
\end{equation}
with $\hat{\lambda}_0 = \frac{5}{4}$, $\hat{\lambda}_1 = 0$ and $\hat{\lambda}_2 = -\frac{3}{4}$, see Ref \onlinecite{VeermanDoelman.2013}. The zeroes of the `fast' transmission function $t_f(\hat{\lambda},\eps)$ are to leading order in $\eps$ given by the discrete spectrum of $\mathcal{L}_f$; in addition, these zeroes are simple. Moreover, the `slow' transmission function $t_s(\hat{\lambda},\eps)$ has a pole of order 1 exactly where $t_f(\hat{\lambda},\eps)$ has a zero, except for the zero associated to $\hat{\lambda}_1 = 0$ -- that is, $t_s$ has a pole for every discrete eigenvalue associated to an even eigenfunction. Hence, the relevant information about the zeroes of $\mathcal{E}$ \eqref{eq:Evansfunction_decomp}, i.e. the point spectrum of $\partial_\xi - \mathcal{A}$ (in $\hat{\lambda}$), is entirely decoded in $t_s$. Note that, although the concept of zero-pole cancellation plays a central role in controllability and the analysis of transfer functions, the occurrence of this phenomenon in the decomposition of the Evans function is completely unrelated to the method of Pyragas control that we employ in this paper; it is solely related to the slow-fast structure of system \eqref{eq:rdstability_Ia_control_xi} and the underlying pulse solution $(u_p,v_p)$.

The `slow' transmission function $t_s$ is, to leading order in $\eps$, given by the following expression according to Ref \onlinecite[Theorem 4.4]{DoelmanVeerman.2015}:
\begin{eqnarray}
 t_s(\hat{\lambda}) &=& C(\hat{\lambda})\Big\{-2 \sqrt{1+\hat{\lambda}+\ell'(0)} \\ &&\quad+ \int_{-\infty}^\infty \frac{\partial}{\partial u}\left[\frac{1}{3} f(u)^2 v_p(\xi)^2 T_o(u)\right] \nonumber\\&&\qquad\quad+ \frac{2}{3} v_\text{in}(\xi;\hat{\lambda}) v_p(\xi) f(u)^2 T_o(u)\,\text{d}\xi\Big\},\nonumber
\end{eqnarray}
evaluated at $u = u_p(0) = u_* = T_o(u_*)$ by \eqref{eq:pulse_existence_condition}. Here, $C(\hat{\lambda})$ can be explicitly computed but is uniformly bounded away from zero, and $v_\text{in}(\hat{\lambda})$ is the unique (bounded) solution to the inhomogeneous Sturm-Liouville problem
\begin{equation}\label{eq:vin_SL-eq}
 \left(\mathcal{L}_f -\hat{\lambda}\right)v_\text{in} = - f'(u_*) v_p(\xi)^2.
\end{equation}
Using the explicit leading order expressions \eqref{eq:pulse_toymodel} for the pulse solution $(u_p,v_p)$, we can further isolate the $u_*$-dependence and rewrite $t_s$ as
\begin{eqnarray}\label{eq:ts}
 t_s(\hat{\lambda}) &=& C(\hat{\lambda})\Bigg\{-2 \sqrt{1+\hat{\lambda}+\ell'(0)} + 2 u_* \bigg[\frac{T_o'(u_*)}{T_o(u_*)} + 2 \frac{f'(u_*)}{f(u_*)} \nonumber\\ &&
 \qquad+ \frac{1}{3}\frac{f'(u_*)}{f(u_*)} \int_{-\infty}^\infty \!\!\!\!\!\hat{v}_p(\xi) \hat{v}_\text{in}(\xi;\hat{\lambda})\,\text{d}\xi\bigg]\Bigg\},
\end{eqnarray}
where $\hat{v}_p(\xi) = f(u_*) v_p(\xi)$ and $\hat{v}_\text{in}$ is the unique bounded solution to 
\begin{equation}\label{eq:app:vinhat_eq}
 \left(\mathcal{L}_f - \hat{\lambda}\right)\hat{v}_\text{in} = - \hat{v}_p(\xi)^2,
\end{equation}
which does not depend on $u_*$. The zeroes of $t_s$ are therefore the (complex) solutions to the equation
\begin{eqnarray}\label{eq:rdstability_Ia_mainequation}
-\frac{T_o'(u_*)}{T_o(u_*)} + \frac{1}{u_*}\sqrt{1+\hat{\lambda}+\ell'(0)} = \\ \frac{f'(u_*)}{f(u_*)}\left(2  + \frac{1}{3} \int_{-\infty}^\infty \!\!\!\!\!\hat{v}_p(\xi) \hat{v}_\text{in}(\xi;\hat{\lambda})\,\text{d}\xi\right).\nonumber
\end{eqnarray}
Note that if $f'(u_*) = 0$, then $v_\text{in}$ \eqref{eq:vin_SL-eq} is trivial; hence no zero-pole cancellation in the Evans function \eqref{eq:Evansfunction_decomp} takes place, and therefore $\lambda_0>0$ is a zero of $\mathcal{E}$, rendering the pulse unstable and not controllable. Hence, assuming the nondegenerate situation $f'(u_*) \neq 0$, we rewrite \eqref{eq:rdstability_Ia_mainequation} as
\begin{equation}\label{eq:rdstability_Ia_control_mainequation_alphabeta}
 \alpha + \beta\sqrt{1+\hat{\lambda}+\ell'(0)} = \int_{-\infty}^\infty \!\!\!\!\!\hat{v}_p(\xi) \hat{v}_\text{in}(\xi;\hat{\lambda})\,\text{d}\xi,
\end{equation}
with $\alpha = -6-3\frac{T_o'(u_*)}{T_o(u_*)}\frac{f(u_*)}{f'(u_*)}$ and $\beta = 3\frac{f(u_*)}{f'(u_*)}\frac{1}{u_*}$; note that $\beta$ can a priori take any value except zero, and $\alpha$ can take any value in $\mathbb{R}$, depending on the properties of the model functions $f$ and $T_o$. The stability control question can now be phrased as follows:

\emph{Given $\alpha, \beta \in \mathbb{R}$, $\beta \neq 0$ fixed, can we choose the control parameter $\ell'(0) < 1$ such, that the (complex) $\hat{\lambda}$-solutions to \eqref{eq:rdstability_Ia_control_mainequation_alphabeta} outside the essential spectrum $\hat{\sigma}_\textrm{\emph{ess}}$ \eqref{eq:rdstability_Ia_control_essspec_hat} all lie to the left of the line $\text{Re }\hat{\lambda} = - \ell'(0)$?}

\subsubsection{Spectral decomposition in the presence of essential spectrum}

Although \eqref{eq:app:vinhat_eq} can in principle be solved by variation of parameters / Green's function methods, the result in terms of integrals over special functions does not easily yield the required insight into the functional behaviour of the right hand side of \eqref{eq:rdstability_Ia_control_mainequation_alphabeta}. Therefore, we use Weyl--Titchmarsh--Kodaira spectral theory to express the integral
\begin{displaymath}
 \int_{-\infty}^\infty \!\!\!\!\!\hat{v}_p(\xi) \hat{v}_\text{in}(\xi;\hat{\lambda})\,\text{d}\xi = \left\langle \hat{v}_\text{in},\hat{v}_p\right\rangle_2
\end{displaymath}
in terms of projections onto eigenfunctions of the operator $\mathcal{L}_f$ \eqref{eq:Lf}. We follow the treatment of Ref \onlinecite[\S 2.18, \S 4.19]{Titchmarsh.1962}. We refer to Ref \onlinecite{HigsonTan.2020} for a modern and highly accessible overview of Weyl--Titchmarsh--Kodaira spectral theory.\\

First, we define $\theta(\xi;\hat{\lambda})$ and $\phi(\xi;\hat{\lambda})$ such that $\left\{\theta,\phi\right\}$ spans the kernel of $\mathcal{L}_f - \hat{\lambda}$, and $\theta(0) = 1$, $\partial_\xi \theta(0) = 0$, $\phi(0) = 0$ and $\partial_\xi \phi(0) = -1$. Since $\mathcal{L}_f$ is symmetric under the reflection $\xi \to -\xi$, $\theta$ is even and $\phi$ is odd as a function of $\xi$. Both $\theta$ and $\phi$ can be expressed as a linear combination of associated Legendre functions $P_3^{-2\sqrt{1+\hat{\lambda}}}(\zeta)$ and $Q_3^{-2\sqrt{1+\hat{\lambda}}}(\zeta)$, with $\zeta = \text{tanh } \frac{\xi}{2}$, yielding
\begin{align*}
 \theta(\xi;\hat{\lambda}) &= \frac{\Gamma\left(4+2\sqrt{1+\hat{\lambda}}\right)}{16\, \hat{\lambda}\sqrt{1+\hat{\lambda}}} \bigg[P_3^{-2\sqrt{1+\hat{\lambda}}}\left(\text{tanh } \tfrac{\xi}{2}\right) \nonumber\\ &\qquad+ P_3^{-2\sqrt{1+\hat{\lambda}}}\left(-\text{tanh } \tfrac{\xi}{2}\right)\bigg],\\
 \phi(\xi;\hat{\lambda}) &= \frac{\Gamma\left(4+2\sqrt{1+\hat{\lambda}}\right)}{16 \left(\hat{\lambda}-\frac{5}{4}\right)\left(\hat{\lambda}+\frac{3}{4}\right)} \bigg[P_3^{-2\sqrt{1+\hat{\lambda}}}\left(\text{tanh } \tfrac{\xi}{2}\right) \nonumber\\&\qquad- P_3^{-2\sqrt{1+\hat{\lambda}}}\left(-\text{tanh } \tfrac{\xi}{2}\right)\bigg].
\end{align*}
Next, for $\hat{\lambda} \in (-\infty,-1)$, we determine $m_2(\hat{\lambda})$ such that the linear combination $\theta(\xi;\hat{\lambda}) + m_2(\hat{\lambda}) \phi(\xi;\hat{\lambda}) \in L^2(0,\infty)$; we find
\begin{equation}
 m_2(\hat{\lambda}) = -\frac{\left(\hat{\lambda}-\frac{5}{4}\right)\left(\hat{\lambda}+\frac{3}{4}\right)}{\hat{\lambda}\sqrt{1+\hat{\lambda}}} = i\frac{\left(\hat{\lambda}-\frac{5}{4}\right)\left(\hat{\lambda}+\frac{3}{4}\right)}{\hat{\lambda}\sqrt{-1-\hat{\lambda}}}.
\end{equation}
Following Ref \onlinecite[\S 4.19]{Titchmarsh.1962}, we conclude that for $f \in L^2(\mathbb{R})$ with $f(\xi) = f(-\xi)$, we have the expansion
\begin{eqnarray}\label{eq:f_Weylexpansion}
 f(\xi) &=& \sum_{i=0}^2 \left\langle \psi_i,f\right\rangle_2 \psi_i(\xi) \\ &&+ \int\displaylimits_{-\infty}^{-1} \frac{\hat{\lambda}\sqrt{-1-\hat{\lambda}}}{2\left(\hat{\lambda}-\frac{5}{4}\right)\left(\hat{\lambda}+\frac{3}{4}\right)} \left\langle \theta(\cdot,\hat{\lambda}),f\right\rangle_2\,\theta(\xi,\hat{\lambda})\;\text{d}\hat{\lambda} \nonumber
\end{eqnarray}
where the ($L^2$-normalised) eigenfunctions $\psi_i(\xi)$ associated to the eigenvalues $\hat{\lambda}_i$ are given by
\begin{align*}
 \psi_0(\xi) &= \frac{1}{4} \sqrt{\frac{15}{2}} \text{sech}^3 \frac{\xi}{2} = \frac{1}{4} \sqrt{\frac{15}{2}}\left(1-\zeta^2\right)^{3/2},\\
 \psi_1(\xi) &= \frac{1}{2} \sqrt{\frac{15}{2}} \text{sech}^2 \frac{\xi}{2}\, \text{tanh }\frac{\xi}{2} = \frac{1}{2} \sqrt{\frac{15}{2}} \zeta\left(1-\zeta^2\right),\\
 \psi_2(\xi) &= \frac{1}{4} \sqrt{\frac{3}{2}} \left(-3+2 \,\text{cosh }\xi\right) \text{sech}^3 \frac{\xi}{2} \\ &= \frac{1}{4}\sqrt{\frac{3}{2}}\left(-1+5\zeta^2\right)\sqrt{1-\zeta^2}.
\end{align*}
Note that $\left\langle f,\psi_1\right\rangle_2 = 0$ since $f(\xi)$ is assumed to be even and $\psi_1(\xi) = -\sqrt{\frac{5}{6}} \partial_\xi \hat{v}_p$ is odd. We apply the expansion \eqref{eq:f_Weylexpansion} to $\hat{v}_p(\xi)$, and subsequently consider the inner product 
\begin{align}
 \left\langle \hat{v}_\text{in}(\cdot,\hat{\lambda}),\hat{v}_p\right\rangle_2 = \sum_{i=0}^2 \left\langle \psi_i,\hat{v}_p\right\rangle_2 \,\left\langle\hat{v}_\text{in}(\cdot,\hat{\lambda}),\psi_i\right\rangle_2 \qquad\qquad\\
 + \int\displaylimits_{-\infty}^{-1}\!\!\! \frac{\hat{\mu}\sqrt{-1-\hat{\mu}}}{2\left(\hat{\mu}-\frac{5}{4}\right)\left(\hat{\mu}+\frac{3}{4}\right)} \left\langle \theta(\cdot,\hat{\mu}),\hat{v}_p\right\rangle_2\left\langle\hat{v}_\text{in}(\cdot,\hat{\lambda}),\theta(\cdot,\hat{\mu})\right\rangle_2\text{d}\hat{\mu},\nonumber
\end{align}
cf. Ref \onlinecite[Theorem 1.8]{HigsonTan.2020}. Now, because
\begin{align*}
 \left\langle \hat{v}_p^2,\psi_i\right\rangle_2 &=& \left\langle -\left[\mathcal{L}_f-\hat{\lambda}\right]\hat{v}_\text{in},\psi_i\right\rangle_2 = \left\langle \hat{v}_\text{in},-\left[\mathcal{L}_f-\hat{\lambda}^*\right]\psi_i\right\rangle\\
 &&= \left\langle\hat{v}_\text{in},-(\hat{\lambda}_i - \hat{\lambda}^*)\psi_i\right\rangle = \left(\hat{\lambda}-\hat{\lambda}_i\right)\left\langle\hat{v}_\text{in},\psi_i\right\rangle
\end{align*}
and equivalently
\begin{displaymath}
 \left\langle \hat{v}_p^2,\theta(\cdot,\hat{\mu})\right\rangle = \left(\hat{\lambda}-\hat{\mu}\right)\left\langle \hat{v}_\text{in}(\cdot,\hat{\lambda}),\theta(\cdot,\hat{\mu})\right\rangle,
\end{displaymath}
we find
\begin{eqnarray}
 \left\langle \hat{v}_\text{in}(\cdot,\hat{\lambda}),\hat{v}_p \right\rangle_2 = \sum_{i=0}^2 \frac{\left\langle \psi_i,\hat{v}_p\right\rangle_2\,\left\langle\hat{v}_p^2,\psi_i\right\rangle_2}{\hat{\lambda}-\hat{\lambda}_i}  \qquad \qquad\\ + \int\displaylimits_{-\infty}^{-1}\!\!\! \frac{\hat{\mu}\sqrt{-1-\hat{\mu}}}{2\left(\hat{\mu}-\frac{5}{4}\right)\left(\hat{\mu}+\frac{3}{4}\right)} \frac{\left\langle \theta(\cdot,\hat{\mu}),\hat{v}_p\right\rangle_2\left\langle\hat{v}_p^2,\theta(\cdot,\hat{\mu})\right\rangle_2}{\hat{\lambda} - \hat{\mu}}\text{d}\hat{\mu}.\nonumber
\end{eqnarray}
All inner products can be calculated explicitly:
\begin{align}
 \left\langle \psi_0,\hat{v}_p\right\rangle_2 &= \frac{9 \pi}{32} \sqrt{\frac{15}{2}},\\
 \left\langle \psi_2,\hat{v}_p\right\rangle_2 &= \frac{3\pi}{32} \sqrt{\frac{3}{2}},\\
 \left\langle\hat{v}_p^2,\psi_0\right\rangle_2 &= \frac{45\pi}{128}\sqrt{\frac{15}{2}},\\
 \left\langle\hat{v}_p^2,\psi_2\right\rangle_2&= -\frac{9\pi}{128}\sqrt{\frac{3}{2}},\\
 \left\langle \theta(\cdot,\hat{\mu}),\hat{v}_p\right\rangle_2 &= -\frac{3 \pi}{4} \sqrt{-1-\hat{\mu}} \;\text{csch } \pi\sqrt{-1-\hat{\mu}},\\
 \left\langle\hat{v}_p^2,\theta(\cdot,\hat{\mu})\right\rangle_2 &= -\frac{3\pi}{4} \hat{\mu}\sqrt{-1-\hat{\mu}}\;\text{csch } \pi\sqrt{-1-\hat{\mu}},
\end{align}
which can be used to write
\begin{align}\label{eq:Rlambda_explicit}
 \left\langle \hat{v}_\text{in}(\cdot,\hat{\lambda}),\hat{v}_p\right\rangle_2 =& \frac{6075\pi^2}{8192} \frac{1}{\hat{\lambda} - \frac{5}{4}} -\frac{81\pi^2}{8192} \frac{1}{\hat{\lambda} + \frac{3}{4}} \\&+ \int\displaylimits_{-\infty}^{-1}\! \frac{9 \pi^2}{32} \frac{\hat{\mu}^2\left(-1-\hat{\mu}\right)^{3/2}}{\left(\hat{\mu}-\frac{5}{4}\right)\left(\hat{\mu}+\frac{3}{4}\right)} \frac{\text{csch}^2 \pi \sqrt{-1-\hat{\mu}}}{\hat{\lambda} - \hat{\mu}}\text{d}\hat{\mu}\nonumber  \\ =&\frac{6075\pi^2}{8192} \frac{1}{\hat{\lambda} - \frac{5}{4}} -\frac{81\pi^2}{8192} \frac{1}{\hat{\lambda} + \frac{3}{4}} \\&+ \int\displaylimits_1^{\infty}\!\frac{9 \pi^2}{16} \frac{\kappa^4(1+\kappa^2)^2}{\left(\kappa^2+\frac{9}{4}\right)\left(\kappa^2+\frac{1}{4}\right)} \frac{\text{csch}^2 \pi \kappa}{\hat{\lambda}+\kappa^2+1}\text{d}\kappa.\nonumber
\end{align}
We can use this information to investigate the solutions to \eqref{eq:rdstability_Ia_control_mainequation_alphabeta}. To this end, we formulate the following Lemma, which summarises relevant properties of the inner product $\left\langle \hat{v}_\text{in}(\cdot,\hat{\lambda}),\hat{v}_p\right\rangle_2$ \eqref{eq:Rlambda_explicit}.

\begin{lem}\label{lem:Rlambda_properties}
 Let $R(\hat{\lambda}) = \left\langle \hat{v}_\text{in}(\cdot,\hat{\lambda}),\hat{v}_p\right\rangle_2$ be as in \eqref{eq:Rlambda_explicit}. Then the following statements hold:
 \begin{enumerate}
  \item[\textrm{(I)}] $\textrm{sgn Im }R(\hat{\lambda}) = - \text{sgn Im }\hat{\lambda}$; 
  \item[\textrm{(II)}] There exists $c>0$ such that, if $\text{Re }\hat{\lambda}>c$, then $\text{Re }R(\hat{\lambda})>0$;
  \item[\textrm{(III)}] For real $\hat{\lambda}>\hat{\lambda}_0$, $R(\hat{\lambda})$ is positive and strictly monotonically decreasing;
  \item[\textrm{(IV)}] There exist $d_{1,2}>0$ such that, if $-d_1<\text{Re }\hat{\lambda}<d_2$, then $\text{Re }R(\hat{\lambda})<0 $.
 \end{enumerate}

\end{lem}

\begin{proof} We define
\begin{align*}
 R_\text{d}(\hat{\lambda}) &= \frac{81 \pi^2}{8192} \left(\frac{75}{\hat{\lambda} - \frac{5}{4}} - \frac{1}{\hat{\lambda} + \frac{3}{4}}\right),\\
 R_\text{c}(\hat{\lambda}) &= \int\displaylimits_1^{\infty}\!\frac{9 \pi^2}{16} \frac{\kappa^4(1+\kappa^2)^2}{\left(\kappa^2+\frac{9}{4}\right)\left(\kappa^2+\frac{1}{4}\right)} \frac{\text{csch}^2 \pi \kappa}{\hat{\lambda}+\kappa^2+1}\text{d}\kappa,
\end{align*}
so $R(\hat{\lambda}) = R_\text{d}(\hat{\lambda}) + R_\text{c}(\hat{\lambda})$.\\
 \textbf{(I)}
 Denoting $\hat{\lambda} = a + i b$, we have $\text{Im } \frac{k}{\hat{\lambda} + m} = b\frac{k}{(a+m)^2 + b^2}$ for $k,m \in \mathbb{R}$. 
 From this observation, it immediately follows that $\text{sgn Im }R_\text{c}(\hat{\lambda}) = - \text{sgn Im }\hat{\lambda}$. 
 For $R_\text{d}$, we estimate $\frac{75}{\left(a-\frac{5}{4}\right)^2 + b^2} - \frac{1}{\left(a+\frac{3}{4}\right)^2 + b^2} > \frac{75-1}{b^2 + \text{max}\left(\left(a-\frac{5}{4}\right)^2,\left(a+\frac{3}{4}\right)^2\right)} > 0$, from which it follows that $\text{sgn Im }R_\text{d}(\hat{\lambda}) = - \text{sgn Im }\hat{\lambda}$.\\
 \textbf{(II)}
 Using the same notation as in (I), we have $\text{Re } \frac{k}{\hat{\lambda} + m} = \frac{k(a+m)}{(a+m)^2 + b^2}$ for $k,m \in \mathbb{R}$.
 From this observation, it immediately follows that $\text{Re }R_\text{c}(\hat{\lambda})>0$ for all $a>0$.
 For $R_\text{d}$, we find that $\frac{75\left(a-\frac{5}{4}\right)}{\left(a-\frac{5}{4}\right)^2 + b^2} - \frac{a+\frac{3}{4}}{\left(a+\frac{3}{4}\right)^2 + b^2} > 0$ for $a > \frac{75 \cdot \frac{5}{4} + 1 \cdot \frac{3}{4}}{75-1} > 0$ and any $b \in \mathbb{R}$.\\
 \textbf{(III)}
 If $\hat{\lambda}_2>\hat{\lambda}_1>0$, then $\frac{k}{\hat{\lambda_2}+m} < \frac{k}{\hat{\lambda_1}+m}$ for $k,m>0$; hence, it follows that $R_\text{c}(\hat{\lambda})$ is strictly monotonically decreasing for real $\hat{\lambda}>0$. 
 Moreover, from the proof of (II), we know that $R_\text{c}(\hat{\lambda})$ is positive for real $\hat{\lambda}>0$.
 For $R_\text{d}(\hat{\lambda})$, we calculate $\frac{\text{d} R_\text{d}}{\text{d}\hat{\lambda}} = \frac{81 \pi^2}{8192} \left( - \frac{75}{\left(\hat{\lambda}-\frac{5}{4}\right)^2} + \frac{1}{\left(\hat{\lambda} + \frac{3}{4}\right)^2}\right) < 0$ for $\hat{\lambda} > \hat{\lambda}_0 = \frac{5}{4}$.
 Moreover, $\lim_{\hat{\lambda} \to \infty} R_\text{d}(\hat{\lambda}) = 0$ and $R_\text{d}(\hat{\lambda}) \to +\infty$ as $\hat{\lambda} \downarrow \hat{\lambda}_0 = \frac{5}{4}$; we conclude that $R_\text{d}$ is positive and strictly monotonically decreasing for real $\hat{\lambda} > \hat{\lambda}_0$.\\
 \textbf{(IV)} 
 For $R_\text{d}$, we find that $\frac{75\left(a-\frac{5}{4}\right)}{\left(a-\frac{5}{4}\right)^2 + b^2} - \frac{a+\frac{3}{4}}{\left(a+\frac{3}{4}\right)^2 + b^2} < 0$ for $-\frac{3}{4} < a < \frac{5}{4}$ and any $b \in \mathbb{R}$. 
 For $R_\text{c}$, we use the estimates $\kappa^2 + \frac{9}{4} > \kappa^2+1$, $\kappa^2+\frac{1}{4} > \kappa^2$ and $\text{Re }\frac{1}{\hat{\lambda}+\kappa^2+1} = \frac{a+\kappa^2+1}{\left(a+\kappa^2+1\right)^2+b^2} \leq \frac{1}{\kappa^2+1}$ when $a\geq0$, to obtain
 \begin{align*}
  \text{Re }R_\text{c}(\hat{\lambda}) &\leq \int\displaylimits_1^{\infty}\!\frac{9 \pi^2}{16} \kappa^2 \text{csch}^2 \pi \kappa\,\text{d}\kappa \\
  &= \frac{9 \pi}{16}\left[3+\text{coth }\pi - \frac{2}{\pi} \log\left(-1+e^{2\pi}\right) + \frac{1}{\pi^2}\text{Li}_2 \left(e^{-2\pi}\right)\right] \\ &\approx 9.05 \cdot 10^{-3}.
 \end{align*}
 The claim follows by continuity of $\text{Re }R = \text{Re }R_\text{d} + \text{Re }R_\text{c}$ in $(a,b)$ for $-\frac{3}{4} < a < \frac{5}{4}$.
 \end{proof}

\subsubsection{Results on proportional feedback control}
Using the results obtained in the previous sections, we can formulate the following Lemma on solutions to equation \eqref{eq:rdstability_Ia_control_mainequation_alphabeta}:

\begin{lem}\label{lem:alphabeta_solutions} Consider equation \eqref{eq:rdstability_Ia_control_mainequation_alphabeta} for $\hat{\lambda} \in \mathbb{C}$ with $\alpha, \beta \in \mathbb{R}$ and $\ell'(0)<1$, where $\hat{v}_\mathrm{in}$ is the unique bounded solution to \eqref{eq:app:vinhat_eq}, and $\hat{v}_p(\xi) = \frac{3}{2}\mathrm{sech}^2 \frac{\xi}{2}$, cf. \eqref{eq:pulse_toymodel}. Then the following statements hold:
\begin{enumerate}
 \item If $\beta>0$, then all solutions $\hat{\lambda}$ to \eqref{eq:rdstability_Ia_control_mainequation_alphabeta} lie on the real axis. 
 \item If $\beta>0$ and $\alpha \leq 0$, then there always exists a real, positive solution $\hat{\lambda}>\hat{\lambda}_0$ to \eqref{eq:rdstability_Ia_control_mainequation_alphabeta} for any $\ell'(0) \in \mathbb{R}$.
 \item If $-\alpha \geq \beta>0$, then for any $\ell'(0) \in \mathbb{R}$, there exists a real, positive solution $\hat{\lambda}>-\ell'(0)$ to \eqref{eq:rdstability_Ia_control_mainequation_alphabeta}.
 \item If $0\leq-\alpha < \beta$, then there exists a value $\ell'(0)<-\hat{\lambda}_0$ such that all solutions $\hat{\lambda}$ to \eqref{eq:rdstability_Ia_control_mainequation_alphabeta} obey $\text{Re }\hat{\lambda} < -\ell(0)$.
 \item If $\beta>0$ and $\alpha>0$, then there exists a value $\ell'(0) < -(1+\hat{\lambda}_0)$ such that \eqref{eq:rdstability_Ia_control_mainequation_alphabeta} has no solutions.
 \item If $\beta<0$, then there exists $\hat{\Lambda}<0$ such that, when $\ell'(0)<\hat{\Lambda}$, all solutions $\hat{\lambda}$ to \eqref{eq:rdstability_Ia_control_mainequation_alphabeta} lie to the left of the vertical line $\left\{\text{Re }\hat{\lambda} = c\right\}$, with $c>0$ independent of $\alpha, \beta$ and $\ell'(0)$. 
\end{enumerate}
\end{lem}

\begin{proof} We define
\begin{displaymath}
 L(\hat{\lambda}) := \alpha + \beta \sqrt{1+\hat{\lambda} + \ell'(0)},
 \end{displaymath}
  so \eqref{eq:rdstability_Ia_control_mainequation_alphabeta} can be written as $L(\hat{\lambda}) = R(\hat{\lambda})$.\\
  \textbf{1. If $\bm{\beta>0}$, then all solutions to \eqref{eq:rdstability_Ia_control_mainequation_alphabeta} lie on the real axis}. 
 For the principal complex square root $\sqrt{z}$, which is defined for all $z$ away from the negative real line, it holds that $\text{Re}\sqrt{z} > 0$ and $\text{sgn Im}\sqrt{z} = \text{sgn Im }z$. 
 It follows that, if $\beta > 0$, $\text{Im }L(\hat{\lambda})$ can only equal $\text{Im }R(\hat{\lambda})$ if both are zero, by Lemma \ref{lem:Rlambda_properties} (I).\\
 \textbf{2. If $\beta>0$ and $\alpha\leq0$, then there always exists a real, positive solution $\hat{\lambda}>\hat{\lambda}_0$ for any $\ell'(0) \in \mathbb{R}$}. 
 By statement 1, we take $\hat{\lambda} \in \mathbb{R}$. 
 Both $R_\text{d}$ and $R_\text{c}$ are positive for sufficiently large $\hat{\lambda}$ by Lemma \ref{lem:Rlambda_properties} (II); moreover, $R(\hat{\lambda}) \to 0$ as $\hat{\lambda} \to \infty$. 
 It follows that $L(\hat{\lambda}) > R(\hat{\lambda})$ for sufficiently large $\hat{\lambda}$, since $\beta>0$. 
 Moreover, from the observation that $R(\hat{\lambda}) \to +\infty$ as $\hat{\lambda} \downarrow \hat{\lambda}_0 = \frac{5}{4}$, combined with the fact that $L(\hat{\lambda})$ is continuous for $\hat{\lambda} \geq -1-\ell'(0)$ and $L(\hat{\lambda}) \to \alpha \leq 0$ as $\hat{\lambda} \downarrow -1-\ell'(0)$, it follows by continuity of $R(\hat{\lambda})$ for $\hat{\lambda} > \hat{\lambda}_0 = \frac{5}{4}$ that there exists a real, positive solution to \eqref{eq:rdstability_Ia_control_mainequation_alphabeta}, that lies to the right of the point $\hat{\lambda} = \hat{\lambda}_0$ and to the right of the point $\hat{\lambda} = -1 - \ell'(0)$.\\
 \textbf{3. If \bm{$-\alpha \geq \beta>0$}, then for any \bm{$\ell'(0) \in \mathbb{R}$}, there exists a real, positive solution \bm{$\hat{\lambda}>-\ell'(0)$} to \eqref{eq:rdstability_Ia_control_mainequation_alphabeta}.}
 Since $R(\hat{\lambda})$ is positive for real $\hat{\lambda}>\hat{\lambda}_0$ by Lemma \ref{lem:Rlambda_properties} (III), the intersection of the graphs $R(\hat{\lambda})$ and $L(\hat{\lambda})$ lies above the horizontal $\hat{\lambda}$-axis. Since $\beta>0$ and $\alpha < 0$, the graph of $L(\hat{\lambda})$ is strictly monotonically increasing on its domain and intersects the horizontal axis before intersecting the graph of $R(\hat{\lambda})$, as $\hat{\lambda}$ increases. 
 Denote the $\hat{\lambda}$-value for which the two graphs intersect as $\hat{\lambda}_*$. 
 From the monotonicity of the graphs, it follows that $\hat{\lambda}_*$ is the largest solution to \eqref{eq:rdstability_Ia_control_mainequation_alphabeta}. 
 Moreover, as $L(\hat{\lambda})$ if and only if $\hat{\lambda} = \left(\frac{\alpha}{\beta}\right)^2-1-\ell'(0)$, we conclude that $\hat{\lambda}_* > \left(\frac{\alpha}{\beta}\right)^2-1-\ell'(0) > -\ell'(0)$ since $-\alpha \geq \beta$.\\
  \textbf{4. If \bm{$0\leq-\alpha < \beta$}, then there exists a value \bm{$\ell'(0)<-\hat{\lambda}_0$} such that all solutions \bm{$\hat{\lambda}$} to \eqref{eq:rdstability_Ia_control_mainequation_alphabeta} obey \bm{$\text{Re }\hat{\lambda} < -\ell(0)$}.} 
  We calculate $L(-1-\ell'(0)) = \alpha \leq 0$ and $L(-\ell'(0)) = \alpha + \beta > 0$. Hence, the largest solution $\hat{\lambda}_*$ to \eqref{eq:rdstability_Ia_control_mainequation_alphabeta} is contained in the interval $(-1-\ell'(0),-\ell'(0))$ if and only if $R(-\ell'(0)) < \alpha + \beta$ by the monotonicity and continuity of $L(\hat{\lambda})$ and $R(\hat{\lambda})$. 
  From Lemma \ref{lem:Rlambda_properties} (III), it follows that $R(-\ell'(0)) < \alpha + \beta$ for sufficiently large $|\ell'(0)|$.\\
  \textbf{5. If \bm{$\beta>0$} and \bm{$\alpha>0$}, then there exists a value \bm{$\ell'(0) < -(1+\hat{\lambda}_0)$} such that \eqref{eq:rdstability_Ia_control_mainequation_alphabeta} has no solutions.} 
  By Lemma \ref{lem:Rlambda_properties} (III), there exists $\hat{k}>\hat{\lambda}_0$ such that $R(\hat{\lambda}) < \alpha$ for all $\hat{\lambda} > \hat{k}$. Choosing $\ell'(0) = -1- \hat{k}$, so that $L(\hat{k}) = L(-1-\ell'(0)) = \alpha$, we see that \eqref{eq:rdstability_Ia_control_mainequation_alphabeta} cannot have any real solutions, which implies that \eqref{eq:rdstability_Ia_control_mainequation_alphabeta} does not have any (complex) solutions by statement 1.\\
 \textbf{6. If \bm{$\beta<0$}, then there exists \bm{$\hat{\Lambda}<0$} such that, when \bm{$\ell'(0)<\hat{\Lambda}$}, all solutions to \eqref{eq:rdstability_Ia_control_mainequation_alphabeta} lie to the left of the vertical line \bm{$\left\{\text{Re }\hat{\lambda} = c\right\}$}, with \bm{$c>0$} independent of \bm{$\alpha, \beta$} and \bm{$\ell'(0)$}.}\\
 Suppose $\alpha \leq 0$, and consider $\hat{\lambda} \in \mathbb{R}$. We can choose $\ell'(0)<-(1+\hat{\lambda}_0)$, making sure that the domain of $L(\hat{\lambda})$ lies entirely to the right of $\hat{\lambda} = \hat{\lambda}_0$. 
 On this domain, we have $L(\hat{\lambda}) < 0$, while $R(\hat{\lambda})>0$, from which we conclude that no real solutions to \eqref{eq:rdstability_Ia_control_mainequation_alphabeta} exist.
 The same reasoning can be extended to complex solutions of \eqref{eq:rdstability_Ia_control_mainequation_alphabeta}, under the same assumption that $\alpha \leq 0$: because $\text{Re}\sqrt{z}>0$ for all $z$ away from the negative real line, we have that $\text{Re }L(\hat{\lambda})<0$ for all $\hat{\lambda}$ on its domain, while there exists a $c>0$ such that $\text{Re }R(\hat{\lambda})>0$ for all $\hat{\lambda}$ to the right of the line $\left\{\text{Re }\hat{\lambda} = c\right\}$ by Lemma \ref{lem:Rlambda_properties} (II).\\
 Now suppose $\alpha>0$. 
 For $\hat{\lambda} \in \mathbb{R}$, the same reasoning as in the case $\alpha \leq 0$ can be applied, choosing $\ell'(0) < \left(-\frac{\alpha}{\beta}\right)^2 -1 - \hat{\lambda}_0$. The reasoning for complex $\hat{\lambda}$ is also analogous to the case $\alpha \leq 0$, now with $\ell'(0) < \left(-\frac{\alpha}{\beta}\right)^2 - 1 - c$, again using Lemma \ref{lem:Rlambda_properties} (II).
 \end{proof}

\subsubsection{Finishing the proof of Theorem \ref{thm:caseIa}}

The results in Lemma \ref{lem:alphabeta_solutions} can now be used prove Theorem \ref{thm:caseIa}, restated here for convenience:
 
\begin{thm*}
Let $0<\eps\ll 1$ be sufficiently small, and assume that $u_*$ is a nondegenerate solution to \eqref{eq:pulse_existence_condition}. Consider the symmetric singular pulse solution $(u_p,v_p)$ to \eqref{eq:rdsystem_explicit}, which is to leading order in $\eps$ given by \eqref{eq:pulse_toymodel}, and introduce $\nu := 2 \frac{f'(u_*)}{f(u_*)} + \frac{T_o'(u_*)}{T_o(u_*)}$.
 \begin{enumerate}
  \item If $f'(u_*) = 0$, then the singular pulse $(u_p,v_p)$ is always unstable for any choice of proportional control function $\ell(v-v_p)$ as implemented in \eqref{eq:rdsystem_explicit_control}.
  \item If $f'(u_*)<0$, then it is possible to choose a proportional control function $\ell(v-v_p)$, as implemented in \eqref{eq:rdsystem_explicit_control}, such that the singular pulse $(u_p,v_p)$ is stable.
  \item Let $f'(u_*)>0$. 
  \begin{enumerate}
  \item If $\rho >\frac{1}{u_*}$, then the singular pulse $(u_p,v_p)$ is always unstable for any choice of proportional control function $\ell(v-v_p)$ as implemented in \eqref{eq:rdsystem_explicit_control}. 
  \item If $\rho < \frac{1}{u_*}$, then it is possible to choose a proportional control function $\ell(v-v_p)$, as implemented in \eqref{eq:rdsystem_explicit_control}, such that the pulse solution $(u_p,v_p)$ is stable.
  \end{enumerate}
 \end{enumerate}
\end{thm*}

\begin{proof}
 The pulse solution $(u_p,v_p)$ is spectrally stable if and only if, for the shifted eigenvalue $\hat{\lambda} = \lambda - \ell'(0)$ \eqref{eq:def_lambdahat}, the condition $\text{Re }\hat{\lambda} < - \ell'(0)$ \eqref{eq:stabcond_lambdahat} is satisfied. Eigenvalues $\hat{\lambda}$ correspond to zeroes of the Evans function $\mathcal{E}(\hat{\lambda},\eps)$ \eqref{eq:Evansfunction_decomp}. These zeroes are, to leading order in $\eps$, given by the solutions to \eqref{eq:rdstability_Ia_mainequation}.\\
 \textbf{Claim 1} follows from the observation that the only solution $v_\text{in}$ to \eqref{eq:vin_SL-eq} is the trivial solution when $f'(u_*) = 0$. Hence, the slow transmission function $t_s$ \eqref{eq:ts} does \emph{not} have a pole at $\hat{\lambda}=\hat{\lambda}_0$; therefore, no zero-pole cancellation in the Evans function \eqref{eq:Evansfunction_decomp} --as detailed in Ref \onlinecite{DoelmanVeerman.2015}-- takes place. It follows that the positive real zero near $\hat{\lambda}_0$ of the fast transmission function $t_f$ is also a zero of the full Evans function $\mathcal{E}$, which means the pulse solution $(u_p,v_p)$ is unstable.\\
 If $f'(u_*) \neq 0$, then \eqref{eq:rdstability_Ia_mainequation} is of the form \eqref{eq:rdstability_Ia_control_mainequation_alphabeta}, with $\alpha = -6-3\frac{T_o'(u_*)}{T_o(u_*)}\frac{f(u_*)}{f'(u_*)}$ and $\beta = 3\frac{f(u_*)}{f'(u_*)}\frac{1}{u_*}$. \\
 When $\beta<0$, it follows from Lemma \ref{lem:alphabeta_solutions} (6) that there is a $c>0$ such that all solutions of \eqref{eq:rdstability_Ia_control_mainequation_alphabeta} lie to the left of the line $\left\{ \text{Re }\hat{\lambda} = c\right\}$. Since $c>0$ is in particular independent of $\ell'(0)$, and the statement holds for all $\ell'(0) < \hat{\Lambda}<0$, we can choose $\ell'(0)<-c$, which implies that all solutions of \eqref{eq:rdstability_Ia_control_mainequation_alphabeta} obey the stability criterion $\text{Re }\hat{\lambda} < - \ell'(0)$. Since $\beta<0$ if and only if $f'(u_*)<0$, this proves \textbf{claim 2}.\\
 When $\alpha>0$ and $\beta>0$, it follows from Lemma \ref{lem:alphabeta_solutions} (5) that \eqref{eq:rdstability_Ia_control_mainequation_alphabeta} does not have any solutions when $\ell'(0)$ is negative and $|\ell'(0)|$ is sufficiently large. As $f'(u_*)\neq 0$, the zeroes of $t_{f}(\hat{\lambda},\eps)$ near $\hat{\lambda}_0$ and $\hat{\lambda}_2$ are cancelled by the poles of $t_s(\hat{\lambda},\eps)$; the remaining zero $\hat{\lambda}_1 = 0$ is stable for $\ell'(0)<0$. It follows that all zeroes the Evans function $\mathcal{E}(\hat{\lambda},\eps)$ \eqref{eq:Evansfunction_decomp} obey $\text{Re }\hat{\lambda} < -\ell'(0)$, which implies that $(u_p,v_p)$ is spectrally stable. The condition $\alpha>0$ is equivalent to $\frac{T_o'(u_*)}{T_o(u_*)}\frac{f(u_*)}{f'(u_*)} > -2$, while $\beta>0$ is equivalent to $\frac{f(u_*)}{f'(u_*)}\frac{1}{u_*}>0$. Since both $u_*>0$ and $f(u_*)>0)$ by assumption, the latter condition is equivalent to $f'(u_*)>0$. This allows us to rewrite the $\alpha>0$-condition as $\frac{T_o'(u_*)}{T_o(u_*)} < -2 \frac{f'(u_*)}{f(u_*)}$, thereby proving the second part of claim 3 for $2 \frac{f'(u_*)}{f(u_*)} + \frac{T_o'(u_*)}{T_o(u_*)} < 0 $.\\
 When $0 \leq - \alpha < \beta$, it follows from Lemma \ref{lem:alphabeta_solutions} (4) that all solutions to \eqref{eq:rdstability_Ia_control_mainequation_alphabeta} obey the stability criterion $\text{Re }\hat{\lambda} < - \ell'(0)$. As before, the condition $\beta>0$ is equivalent to $f'(u_*)$, while the condition $\beta > - \alpha \geq 0$ can be rewritten as $0 \leq 2 \frac{f'(u_*)}{f(u_*)} + \frac{T_o'(u_*)}{T_o(u_*)} < \frac{1}{u_*}$, which finalises the proof of \textbf{claim 3 b)}.\\
 When $ - \alpha > \beta > 0$, it follows from Lemma \ref{lem:alphabeta_solutions} (3) that there always exists a solution to \eqref{eq:rdstability_Ia_control_mainequation_alphabeta} that does not satisfy the stability criterion $\text{Re }\hat{\lambda} < - \ell'(0)$, regardless of the value of $\ell'(0)$. This implies that, under the condition $ - \alpha > \beta > 0$, the pulse $(u_p,v_p)$ is unstable for any choice of the control function $\ell$. The condition $\beta > 0$ is equivalent to $f'(u_*)>0$, while the condition $-\alpha > \beta$ can be rewritten as $\frac{f'(u_*)}{f(u_*)} + \frac{T_o'(u_*)}{T_o(u_*)} > \frac{1}{u_*}$, which proves \textbf{claim 3 a)}.
\end{proof}

\nocite{*}
\bibliography{PatternControl_references}
\bibliographystyle{test4}

\end{document}